\documentclass[a4paper]{amsart}

\usepackage{amsmath,amsthm,amssymb,latexsym,epic,bbm,comment, mathrsfs}

\usepackage{graphicx,enumerate,stmaryrd,color}

\usepackage[all,2cell]{xy}

\xyoption{2cell}

\newtheorem{theorem}{Theorem}

\newtheorem{lemma}[theorem]{Lemma}

\newtheorem{proposition}[theorem]{Proposition}

\usepackage[all]{xy}

\usepackage[active]{srcltx}

\usepackage[parfill]{parskip}

\usepackage{enumerate}

\newcommand{\tto}{\twoheadrightarrow}

\begin{document}

\title[Bimodules over Nakayama algebra]{Cell structure of bimodules over\\ radical square zero Nakayama algebras}

\author{Helena Jonsson}

\begin{abstract}
In this paper we describe the combinatorics of the cell structure of the
tensor category of bimodules over a radical square zero Nakayama algebra.
This accounts to an explicit description of left, right and two-sided cells.
\end{abstract}

\maketitle



\section{Introduction and description of the results}\label{s1}
Let $\Bbbk$ be an algebraically closed field of characteristic zero. For a positive
integer $n>1$, let $Q_n$ denote the quotient of the path algebra of the quiver
\begin{displaymath}
\xymatrix{
&&&1\ar[dlll]_{\alpha _{1}}&&&\\
2\ar[r]_{\alpha _2}&3\ar[r]_{\alpha _3}&4\ar[r]_{\alpha _4}&
\dots\ar[r]_{\alpha _{n-2}}&n-1\ar[r]_{\alpha _{n-1}}&n\ar[ull]_{\alpha _{n}} 
} 
\end{displaymath}
of affine Coxeter type $\tilde{A}_{n-1}$ by the relations that all paths of length two in this
quiver are equal to zero. We note that we compose paths from right to left.

We also denote by $Q_1$ the algebra $\Bbbk[x]/(x^2)$ of dual numbers
over $\Bbbk$. This algebra is isomorphic to the quotient of the path algebra of the quiver
\begin{displaymath}
\xymatrix{
0\ar@(ur,dr)[]^{\alpha_0}
} 
\end{displaymath}
by the relation that the path of length two in this quiver is equal to zero.
For each positive integer $n$, the algebra  $Q_n$ is a Nakayama algebra, see \cite{Na}.
In what follows we fix $n$ and set $A=Q_n$.

Consider the tensor category $A$-mod-$A$ of all finite dimensional $A$-$A$-bimodules
and let $\mathcal{S}$ denote the set of isomorphism classes of indecomposable objects 
in $A$-mod-$A$. Note that $\mathcal{S}$ is an infinite set.
For an $A$-$A$-bimodule $X$, we will denote by $[X]$ the class of $X$ in $\mathcal{S}$.
By \cite[Section~3]{MM2}, the set $\mathcal{S}$ has the natural structure
of a multisemigroup, cf. \cite{KuM}, defined as follows: 
for two indecomposable $A$-$A$-bimodules $X$ and $Y$, we have
\begin{displaymath}
[X]\star[Y]:=\{[Z]\in \mathcal{S}\,:\, Z\text{ is isomorphic to a direct summand of }X\otimes_A Y\}.
\end{displaymath}
The basic combinatorial structure of a multisemigroup (or, as a special case, of a semigroup)
is encoded into the so-called  Green's relations, see \cite{Gr,GaMa,KuM}. For $\mathcal{S}$,
these Green's relations are defined as follows.

For two indecomposable $A$-$A$-bimodules $X$ and $Y$, we write $[X]\geq_L [Y]$ provided that there is
an indecomposable $A$-$A$-bimodules $Z$ such that $[X]\in [Z]\star [Y]$. The relation $\geq_L$ 
is a partial pre-order on $\mathcal{S}$, called the {\em left preorder} and equivalence classes for $\geq_L$
are called {\em left cells}. Similarly one defines the {\em right preorder} $\geq_R$ and 
{\em right cells} using $[Y]\star [Z]$, and also the {\em two-sided preorder} $\geq_J$ and 
{\em two-sided cells} using $[Z]\star[Y]\star [Z']$. We will abuse the language and often speak about
cells of bimodules (and not of isomorphism classes of bimodules).

The main aim of the present paper is an explicit description of left, right and two-sided cells in 
$\mathcal{S}$. As the algebra $A\otimes_{\Bbbk}A^{\mathrm{op}}$ is special biserial, cf. \cite{BR,WW},
all indecomposable $A$-$A$-bimodules split into three types, see Section~\ref{s2} for details:
\begin{itemize}
\item string bimodules;

\item band bimodules;

\item non-uniserial projective-injective bimodules. 
\end{itemize}

Following \cite{Jo,MZ}, for a string bimodule $X$ we consider a certain invariant $\mathbf{v}(X)$
called the {\em number of valleys} in the graph of $X$. The structure of two-sided cells is
described by the following:

\begin{theorem}\label{thm1}
{\hspace{2mm}}

\begin{enumerate}[$($a$)$]

\item\label{thm1.1} All band bimodules form a two-sided cell denoted  $\mathcal{J}_{\mathrm{band}}$.

\item\label{thm1.2} For each positive integer $k$, all string bimodules with $\mathbf{v}(X)=k$ form a
two-sided cell denoted  $\mathcal{J}_{k}$.

\item\label{thm1.3} All $\Bbbk$-split bimodules in the sense of \cite{MMZ} form 
a two-sided cell denoted  $\mathcal{J}_{\mathrm{split}}$.

\item\label{thm1.4} All string bimodules with $\mathbf{v}(X)=0$ which are not $\Bbbk$-split form a
two-sided cell denoted  $\mathcal{J}_{0}$.

\item\label{thm1.5} All two-sided cells are linearly ordered as follows:

\begin{displaymath}
\mathcal{J}_{\mathrm{split}}\geq_J \mathcal{J}_{0}\geq_J \mathcal{J}_{1}\geq_J \mathcal{J}_{2}
\geq_J \dots \geq_J  \mathcal{J}_{\mathrm{band}}.
\end{displaymath}
\end{enumerate}
\end{theorem}

Also following \cite{Jo,MZ}, all  non-$\Bbbk$-split
string bimodules can be divided into four different types, $M$, $W$, $N$ or $S$,
depending on the action graph. For each non-$\Bbbk$-split string bimodule, 
in Subsection~\ref{s2.4} 
we introduce three invariants: the {\em initial vertex}, {\em width} and {\em hight}.
The structure of left and right cells is then given by the following.

\begin{theorem}\label{thm2}

{\hspace{2mm}}

\begin{enumerate}[$($a$)$]

\item\label{thm2.1} The two-sided cell $\mathcal{J}_{\mathrm{band}}$ is also a left and a right cell.
\item\label{thm2.2} Left cells in  $\mathcal{J}_{\mathrm{split}}$ are indexed by indecomposable 
right $A$-modules. For an indecomposable  right $A$-module $N$, the left cell of $N$ consists of all
$M\otimes_{\Bbbk}N$, where $M$ is an indecomposable left $A$-module.
\item\label{thm2.3} Right cells in  $\mathcal{J}_{\mathrm{split}}$ are indexed by indecomposable 
left $A$-modules. For an indecomposable left $A$-module $M$, the right cell of $M$ consists of all
$M\otimes_{\Bbbk}N$, where $N$ is an indecomposable right $A$-module.
\item\label{thm2.4} For a non-negative integer $k$, a left cell in $\mathcal{J}_k$
consists of all bimodules in $\mathcal{J}_k$ which has the same second coordinate of the 
initial vertex and the same width. Two bimodules in the same left cell are necessary either 
of the same type or  of type $M$ or $N$ alternatively of type $W$ or $S$.
\item\label{thm2.5} For a non-negative integer $k$, a right cell in $\mathcal{J}_k$
consists of all bimodules in $\mathcal{J}_k$ which has the same first coordinate of the
initial vertex and the same height. Two bimodules in the same right same are necessarily either 
of the same type of or type $M$ or $S$ alternatively of type $W$ or $N$.
\item\label{thm2.6} All two-sided cells, with the exception of $\mathcal{J}_{\mathrm{band}}$, 
are strongly regular in the sense of \cite[Subsection~4.8]{MM1}.
\end{enumerate}
\end{theorem}

The paper is organized as follows: In Section~\ref{s2} we recall the classification of indecomposable
$A$-$A$-bimodules and collect all necessary preliminaries. After some preliminary computations of 
tensor products in Section~\ref{s3}, we  prove Theorem~\ref{thm1}
in Section~\ref{s4} and Theorem~\ref{thm2} in Section~\ref{s5}.

For the special case of $A$ being the algebra of dual numbers, some of the results of this 
paper were obtained in author's Master Thesis \cite{Jo}. 
In fact, in the case of dual numbers, \cite{Jo} provides detailed (very technical) explicit 
formulae for decomposition of tensor product of indecomposable bimodules which we decided not
to include in the present paper.
The results of this paper can also be seen as a generalization and an extension of the results of
\cite{MZ} which describes the cell combinatorics of the tensor category of bimodules over
the radical square zero quotient of a uniformly oriented Dynkin quiver of type $A_n$.
The results of this paper can also be compared with the results of \cite{GM1,GM2,Fo}.
\vspace{5mm}

\textbf{Acknowledgements:} This research is partially supported by
G{\"o}ran Gustafsson Stiftelse.

\section{Indecomposable $A$-$A$-bimodules}\label{s2}

\subsection{Quiver and relations for $A$-$A$-bimodules}\label{s2.1}

The category $A\text{-mod-}A$ is equivalent to the category $A\otimes A^{\text{op}}$-mod 
of finitely generated left $A\otimes A^{\text{op}}$-modules. The latter is equivalent 
to the category of modules over the path algebra of the discrete torus given by the 
following diagram, where we identify the first row with the last row and the first 
column with the last column, modulo the relations that 
\begin{itemize}
\item the composition of any 
two vertical arrows or horizontal arrows is $0$,
\item all squares commute.
\end{itemize}

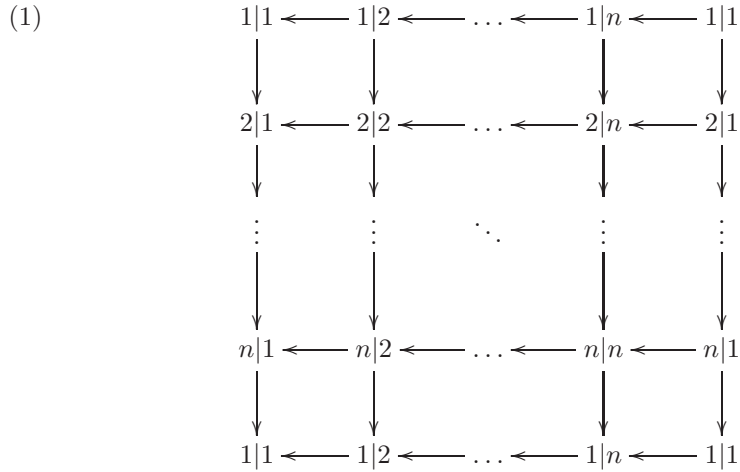
\begin{equation}\label{eq1}
\xymatrix{  
1|1 \ar[d] & 1|2 \ar[l]\ar[d] & \ldots \ar[l] & 1|n \ar[l]\ar[d] & 1|1 \ar[l]\ar[d]\\
2|1 \ar[d] & 2|2 \ar[l]\ar[d] & \ldots \ar[l] & 2|n \ar[l]\ar[d] & 2|1 \ar[l]\ar[d]\\
\vdots \ar[d] & \vdots \ar[d] & \ddots & \vdots \ar[d] & \vdots \ar[d] \\
n|1 \ar[d] & n|2 \ar[l]\ar[d] & \ldots \ar[l]& n|n \ar[l]\ar[d] & n|1 \ar[l]\ar[d]\\
1|1 & 1|2 \ar[l] & \ldots \ar[l]& 1|n \ar[l] & 1|1 \ar[l]}
\end{equation}

Here indices of nodes are identified modulo $n$ (i.e. are in $\mathbb{Z}_n$),
vertical arrows correspond to the left component $A$ and horizontal arrows 
correspond to the right component $A^{\mathrm{op}}$.

We denote by $\varepsilon_i$ the primitive idempotent of $A$ corresponding to the vertex $i$.

If an $A$-$A$-bimodule $X$ is considered as a representation of \eqref{eq1}, we denote
by $X_{i|j}$ the value of $X$ at the vertex $i|j$.

\subsection{$\Bbbk$-split $A$-$A$-bimodules}\label{s2.2}

A $\Bbbk$-split $A$-$A$-bimodule, cf. \cite{MMZ}, is a bimodule of the form
$M\otimes_{\Bbbk} N$, where $M$ is a left $A$-module and $N$ is a right $A$-module. 
The bimodule $M\otimes_{\Bbbk} N$ is indecomposable if and only if both $M$ and $N$ are
indecomposable. The additive closure in $A$-mod-$A$ of all $\Bbbk$-split $A$-$A$-bimodules 
is the unique minimal tensor ideal. Therefore all indecomposable $\Bbbk$-split $A$-$A$-bimodules
belong to the same two-sided cell, denoted $\mathcal{J}_{\mathrm{split}}$.
This proves Theorem~\ref{thm1}\eqref{thm1.3}.

Note that 
\begin{displaymath}
M'\otimes_{\Bbbk}A\otimes_A  M\otimes_{\Bbbk} N\cong  \big(M'\otimes_{\Bbbk}N\big)^{\oplus \dim_{\Bbbk} M}.
\end{displaymath}
This implies that left cells in  $\mathcal{J}_{\mathrm{split}}$ are of the form
\begin{displaymath}
\{[M\otimes_{\Bbbk} N]\,:\, M\text{ is indecomposable}\} 
\end{displaymath}
and $N$ is fixed. This implies Theorem~\ref{thm2}\eqref{thm2.2} and
Theorem~\ref{thm2}\eqref{thm2.3} is obtained similarly.

As a special case of $\Bbbk$-split $A$-$A$-bimodules, we have 
projective-injective bimodules (they all are non-uniserial) which correspond to the 
commuting squares in the diagram \eqref{eq1}. For $n>1$ and fixed $i$ and $j$, the 
non-zero part of the corresponding 
projective-injective bimodule $P(i|j)\cong I(i+1|j-1)$ realized as a representation of the quiver
\eqref{eq1} looks as follows:
\begin{displaymath}
\xymatrix{  \Bbbk_{i|j-1} \ar[d]\ar[d]_{\mathrm{id}} & \Bbbk_{i|j} \ar[l]_{\mathrm{id}}\ar[d]^{\mathrm{id}} \\
\Bbbk_{i+1|j-1} & \Bbbk_{i+1|j} \ar[l]^{\mathrm{id}}}
\end{displaymath}
Here $\Bbbk_{s|t}$ denotes a copy of $\Bbbk$ at the vertex $s|t$. For $n=1$,
we have the following picture:
\begin{displaymath}
\xymatrix{  \Bbbk^4_{1|1} \ar[d]\ar[d]_{\varphi} & \Bbbk^4_{1|2} \ar[l]_{\psi}\ar[d]^{\varphi} \\
\Bbbk^4_{2|1} & \Bbbk^4_{2|2} \ar[l]^{\psi}}
\end{displaymath}
where $\varphi$ and $\psi$ are given by the respective matrices
\begin{displaymath}
\left(\begin{array}{cccc}0&0&0&0\\1&0&0&0\\0&0&0&0\\0&0&1&0\end{array}\right)\qquad\text{ and }\qquad
\left(\begin{array}{cccc}0&0&0&0\\0&0&0&0\\1&0&0&0\\0&1&0&0\end{array}\right)
\end{displaymath}

Another special case of $\Bbbk$-split $A$-$A$-bimodules are simple $A$-$A$-bimodules $L(i|j)$,
for fixed $i$ and $j$. The bimodule $L(i|j)$ is one-dimensional and the 
non-zero part of $L(i|j)$ realized as a representation of the quiver
\eqref{eq1} looks as $\Bbbk_{i|j}$.

Finally, we also have the bimodules $S_{i|j}^{(0)}$ and $N_{i|j}^{(0)}$ given by their respective
non-zero parts of realizations as a representation of the quiver \eqref{eq1}: 
\begin{displaymath}
\xymatrix{\Bbbk_{i|j-1} & \Bbbk_{i|j} \ar[l]_{\mathrm{id}}}\quad\quad\quad \text{and} 
\quad\quad\quad \xymatrix{\Bbbk_{i|j} \ar[d]^{\mathrm{id}} \\ \Bbbk_{i+1|j},}
\end{displaymath}
for $n>1$ and
\begin{displaymath}
\xymatrix{  \Bbbk^2_{1|1} \ar[d]\ar[d]_{0} & \Bbbk^2_{1|2} \ar[l]_{\varphi}\ar[d]^{0} \\
\Bbbk^2_{2|1} & \Bbbk^2_{2|2} \ar[l]^{\varphi}}\quad\quad \text{and}\quad\quad
\xymatrix{  \Bbbk^2_{1|1} \ar[d]\ar[d]_{\varphi} & \Bbbk^2_{1|2} \ar[l]_{0}\ar[d]^{\varphi} \\
\Bbbk^2_{2|1} & \Bbbk^2_{2|2} \ar[l]^{0}}
\end{displaymath}
where $\varphi$ denotes the matrix
\begin{equation}\label{eq2}
\left(\begin{array}{cc}0&0\\1&0\end{array}\right),
\end{equation}
for $n=1$.

\subsection{Band $A$-$A$-bimodules}\label{s2.3}
Band $A$-$A$-bimodules, as classified in \cite{BR,WW}, are $A$-$A$-bimodules $B(k,m,\lambda)$,
where $k\in\mathbb{Z}_n$, $m$ is a positive integer and $\lambda$ is a non-zero number.
We now recall their construction.
The foundational band $A$-$A$-bimodule is the regular $A$-$A$-bimodule ${}_AA_A=B(1,1,1)$.
For $n=1$ and $n=2$, here are the respective realizations of this bimodule as a representation of 
\eqref{eq1} (here $\varphi$ is given by \eqref{eq2}):
\begin{displaymath}
\xymatrix{  \Bbbk^2_{1|1} \ar[d]\ar[d]_{\varphi} & \Bbbk^2_{1|2} \ar[l]_{\varphi}\ar[d]^{\varphi} \\
\Bbbk^2_{2|1} & \Bbbk^2_{2|2} \ar[l]^{\varphi}}\quad\quad\text{ and }\quad\quad
\xymatrix{
\Bbbk_{1|1} \ar[d]\ar[d]_{\mathrm{id}} & \Bbbk_{1|2} \ar[l]_{0}\ar[d]^{0}
& \Bbbk_{1|3} \ar[l]_{\mathrm{id}}\ar[d]^{\mathrm{id}} \\
\Bbbk_{2|1}\ar[d]_{0} & \Bbbk_{2|2}\ar[d]_{\mathrm{id}} \ar[l]^{\mathrm{id}}
& \Bbbk_{2|3} \ar[d]_{0}\ar[l]^{0}\\
\Bbbk_{3|1} & \Bbbk_{3|2} \ar[l]^{0}& \Bbbk_{3|3} \ar[l]^{\mathrm{id}}
}
\end{displaymath}

For a non-zero $\lambda \in \Bbbk$, denote by $\eta_\lambda$ the automorphism of 
$A$ given by multiplying $\alpha_1$ with $\lambda$ and keeping all other basis elements of $A$ untouched.
For an $A$-$A$-bimodule $X$, 
we can consider the $A$-$A$-bimodule $X^{\eta_\lambda}$ in which the right action of $A$ is twisted
by $\eta_\lambda$. Then we have
\begin{displaymath}
B(1,1,\lambda)=B(1,1,1)^{\eta_\lambda}.
\end{displaymath}
For $n=2$, the realizations of $B(1,1,\lambda)$ as a representation of 
\eqref{eq1} is as follows:
\begin{displaymath}
\xymatrix{
\Bbbk_{1|1} \ar[d]\ar[d]_{\mathrm{id}} & \Bbbk_{1|2} \ar[l]_{0}\ar[d]^{0}
& \Bbbk_{1|3} \ar[l]_{\mathrm{id}}\ar[d]^{\mathrm{id}} \\
\Bbbk_{2|1}\ar[d]_{0} & \Bbbk_{2|2}\ar[d]_{\mathrm{id}} \ar[l]^{\lambda\mathrm{id}}
& \Bbbk_{2|3} \ar[d]_{0}\ar[l]^{0}\\
\Bbbk_{3|1} & \Bbbk_{3|2} \ar[l]^{0}& \Bbbk_{3|3} \ar[l]^{\mathrm{id}}
}
\end{displaymath}

The bimodule $B(1,m,\lambda )$ fits into a short exact sequence
\begin{displaymath}
0\to B(1,1,\lambda )\to B(1,m,\lambda )\to B(1,m-1,\lambda )\to 0.
\end{displaymath}
The definition of $B(1,m,\lambda )$ is best explained by the following example.
For $n=2$, the realizations of $B(1,m,\lambda)$ as a representation of 
\eqref{eq1} is as follows:
\begin{displaymath}
\xymatrix{
\Bbbk^m_{1|1} \ar[d]\ar[d]_{\mathrm{id}} & \Bbbk^m_{1|2} \ar[l]_{0}\ar[d]^{0}
& \Bbbk^m_{1|3} \ar[l]_{\mathrm{id}}\ar[d]^{\mathrm{id}} \\
\Bbbk^m_{2|1}\ar[d]_{0} & \Bbbk^m_{2|2}\ar[d]_{\mathrm{id}} \ar[l]^{J_m(\lambda)}
& \Bbbk^m_{2|3} \ar[d]_{0}\ar[l]^{0}\\
\Bbbk^m_{3|1} & \Bbbk^m_{3|2} \ar[l]^{0}& \Bbbk^m_{3|3} \ar[l]^{\mathrm{id}}
}
\end{displaymath}
where $J_m(\lambda)$ denotes the Jordan $m\times m$-cell with eigenvalue $\lambda$.
For $n=1$, the bimodule $B(1,m,\lambda)$ is given by
\begin{displaymath}
\xymatrix{  \Bbbk^{2m}_{1|1} \ar[d]\ar[d]_{\psi} & \Bbbk^{2m}_{1|2} \ar[l]_{\varphi}\ar[d]^{\psi} \\
\Bbbk^{2m}_{2|1} & \Bbbk^{2m}_{2|2} \ar[l]^{\varphi}}
\end{displaymath}
where $\psi$ and $\varphi$ are respectively given by the matrices 
\begin{displaymath}
\left(\begin{array}{cc}0&0\\E&0\end{array}\right)\quad\quad\text{ and }\quad\quad
\left(\begin{array}{cc}0&0\\J_m(\lambda)&0\end{array}\right)
\end{displaymath}
(here each matrix is a $2\times 2$ block matrix with $m\times m$ blocks and $E$ is the identity matrix).
Note that $B(1,m,\lambda )\cong B(1,m,1) ^{\eta _\lambda}$ for any $m$ and $\lambda$.

Finally, let $\theta$ denote the automorphism of $A$ given by the elementary rotation of the quiver which
sends each $\varepsilon_i$ to $\varepsilon_{i+1}$ and each $\alpha_i$ to $\alpha_{i+1}$.
Clearly, $\theta^n=\mathrm{Id}$. For an $A$-$A$-bimodule $X$ and any $k\in \mathbb{Z}_n$, 
we can consider the $A$-$A$-bimodule $X^{\theta^k}$ in which the right action of $A$ is twisted
by $\theta^k$. Then, for all $m$ and $\lambda$, we have
\begin{displaymath}
B(k,m,\lambda)\cong B(1,m,\lambda)^{\theta^{k-1}}.
\end{displaymath}
For $n=2$, the realizations of $B(2,1,1)$ as a representation of 
\eqref{eq1} is as follows:
\begin{displaymath}
\xymatrix{
\Bbbk_{1|1} \ar[d]\ar[d]_{0} & \Bbbk_{1|2} \ar[l]_{\mathrm{id}}\ar[d]^{\mathrm{id}}
& \Bbbk_{1|3} \ar[l]_{0}\ar[d]^{0} \\
\Bbbk_{2|1}\ar[d]_{\mathrm{id}} & \Bbbk_{2|2}\ar[d]_{0} \ar[l]^{0}
& \Bbbk_{2|3} \ar[d]_{\mathrm{id}}\ar[l]^{\mathrm{id}}\\
\Bbbk_{3|1} & \Bbbk_{3|2} \ar[l]^{\mathrm{id}}& \Bbbk_{3|3} \ar[l]^{0}
}
\end{displaymath}

Hence, any band $A$-$A$-bimodule can be constructed from $B(1,1,1)$ using extensions and twists by 
$\theta^k$ and $\eta_\lambda$.

\subsection{String $A$-$A$-bimodules}\label{s2.4}

String bimodules are best understood using covering techniques, see e.g. \cite{Ri}.
Consider the infinite quiver 
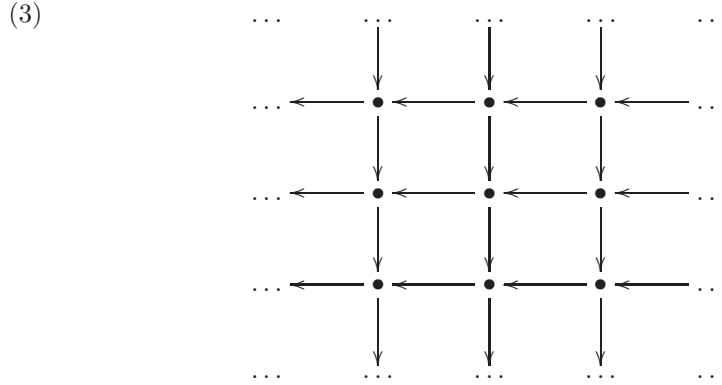
\begin{equation}\label{eq3}
\xymatrix{
\dots&\dots\ar[d]&\dots\ar[d]&\dots\ar[d]&\dots\\
\dots&\bullet\ar[d]\ar[l]&\bullet\ar[d]\ar[l]&\bullet\ar[d]\ar[l]&\dots\ar[l]\\
\dots&\bullet\ar[d]\ar[l]&\bullet\ar[d]\ar[l]&\bullet\ar[d]\ar[l]&\dots\ar[l]\\
\dots&\bullet\ar[d]\ar[l]&\bullet\ar[d]\ar[l]&\bullet\ar[d]\ar[l]&\dots\ar[l]\\
\dots&\dots&\dots&\dots&\dots\\
}
\end{equation}
with the vertex set $\mathbb{Z}^2$ and the same relations as in \eqref{eq1},
that is all squares commute and the composition of any two horizontal or any two vertical
arrows is zero. The group $\mathbb{Z}^2$ acts on this quiver such that the standard
generators act by horizontal and vertical shifts by $n$, respectively. In this way the
above quiver is a covering of \eqref{eq1} in the sense of \cite{Ri}.
We denote by $\Theta$ the usual functor from the category of finite dimensional
modules over \eqref{eq3} to the category of finite dimensional modules over \eqref{eq1}.
All indecomposable string $A$-$A$-bimodules are obtained from indecomposable
finite dimensional representations of \eqref{eq3} using $\Theta$.

The relevant representations of \eqref{eq3} are denoted $V(x,c,l)$, where the parameter
$v=i|j$ is a vertex of \eqref{eq3} such that $1\leq i,j\leq n$, 
the parameter $c$ (the {\em course})  takes values in $\{r,d\}$ where $r$ is a shorthand
for ``right'' and $d$ is a shorthand for ``down'', and, finally, $l$, the {\em length}, 
is a non-negative integer (if $l=0$, then the parameter $c$ has no value). 
The representation $V(x,c,l)$ has total dimension $l+1$ and is constructed as follows:
\begin{itemize}

\item start at the vertex $v$;

\item choose the initial course (or direction) given by $c$;

\item make an alternating right/down walk from $v$ of length $l$;

\item put a one-dimensional vector space at each vertex of this walk;

\item put the identity operator at each arrow of this walk;

\item set all the remaining vertices and arrows to zero. 

\end{itemize}
For example, the non-zero part of the module $V(1|2,r,3)$ is as follows:
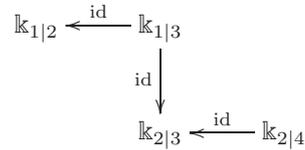
\begin{displaymath}
\xymatrix{
\Bbbk_{1|2}&\Bbbk_{1|3}\ar[l]_{\mathrm{id}}\ar[d]_{\mathrm{id}}&\\
&\Bbbk_{2|3}&\Bbbk_{2|4}\ar[l]_{\mathrm{id}}
}
\end{displaymath}

Note that $\Theta(V(v,c,l))$ is $\Bbbk$-split if $l\leq 1$. Therefore
from now one we assume that $l\geq 2$. We will say that
\begin{itemize}

\item $V(v,c,l)$ if {\em of type $M$} if $c=r$ and $l$ is even;

\item $V(v,c,l)$ if {\em of type $N$} if $c=r$ and $l$ is odd;

\item $V(v,c,l)$ if {\em of type $W$} if $c=d$ and $l$ is even;

\item $V(v,c,l)$ if {\em of type $S$} if $c=d$ and $l$ is odd.

\end{itemize}

Here are examples of modules of respective types $M$ and $N$:
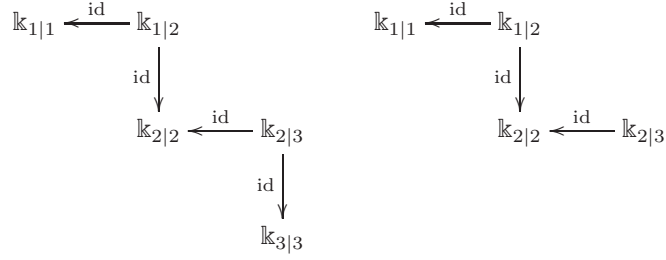
\begin{displaymath}
\xymatrix{
\Bbbk_{1|1}&\Bbbk_{1|2}\ar[l]_{\mathrm{id}}\ar[d]_{\mathrm{id}}&\\
&\Bbbk_{2|2}&\Bbbk_{2|3}\ar[l]_{\mathrm{id}}\ar[d]_{\mathrm{id}}\\
&&\Bbbk_{3|3}\\
}\qquad
\xymatrix{
\Bbbk_{1|1}&\Bbbk_{1|2}\ar[l]_{\mathrm{id}}\ar[d]_{\mathrm{id}}&\\
&\Bbbk_{2|2}&\Bbbk_{2|3}\ar[l]_{\mathrm{id}}
}
\end{displaymath}

Here are examples of modules of respective types $W$ and $S$:

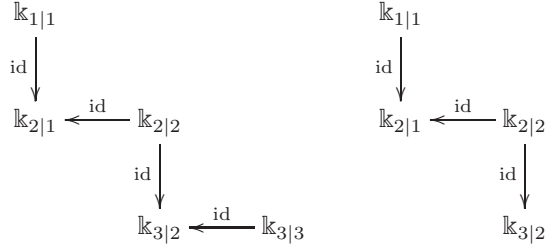
\begin{displaymath}
\xymatrix{
\Bbbk_{1|1}\ar[d]_{\mathrm{id}}&&\\
\Bbbk_{2|1}&\Bbbk_{2|2}\ar[l]_{\mathrm{id}}\ar[d]_{\mathrm{id}}&\\
&\Bbbk_{3|2}&\Bbbk_{3|3}\ar[l]_{\mathrm{id}}\\
}\qquad
\xymatrix{
\Bbbk_{1|1}\ar[d]_{\mathrm{id}}&\\
\Bbbk_{2|1}&\Bbbk_{2|2}\ar[l]_{\mathrm{id}}\ar[d]_{\mathrm{id}}\\
&\Bbbk_{3|2}\\
}
\end{displaymath}

The vertex $v$ of $V(v,c,l)$ is called the {\em initial vertex}. The {\em width}
of $V(v,c,l)$ is the number of non-zero columns. For modules of type $M$ and $N$,
the width is given by $\lceil\frac{l+2}{2}\rceil$. For modules of type $W$ and $S$,
the width is given by $\lfloor\frac{l+2}{2}\rfloor$.

The {\em height} of $V(v,c,l)$ is the number of non-zero rows. For modules of type $M$ and $N$,
the height is given by $\lfloor\frac{l+2}{2}\rfloor$. For modules of type $W$ and $S$,
the height is given by $\lceil\frac{l+2}{2}\rceil$.

A {\em valley} of a module is a vertex of the form
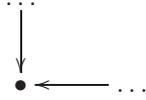
\begin{displaymath}
\xymatrix{\dots\ar[d]&\\\bullet&\ar[l]\dots} 
\end{displaymath}
where both incoming arrows are non-zero. The number of valleys in $V$ is denoted
$\mathbf{v}(V)$. We have
\begin{displaymath}
\mathbf{v}(V(v,c,l))=
\begin{cases}

\lfloor\frac{l-1}{2}\rfloor,& c=r;\\ 

\lfloor\frac{l}{2}\rfloor,& c=d. 
\end{cases}
\end{displaymath}

Note that the module $V(v,c,l)$ is uniquely determined by its type, $v$ and the
number of valleys. For any type $X\in\{M,N,W,S\}$, we denote by 
$X^{(k)}_{i\vert j}$ the $A$-$A$-bimodule obtained using $\Theta$ from the
corresponding representation of \eqref{eq3} of type $X$, the initial vertex $i|j$
and having $k$ valleys.

From \cite{BR,WW} it follows that any indecomposable $A$-$A$-bimodule is isomorphic to one of the
bimodules defined in this section (i.e. is $\Bbbk$-split, a band or a string bimodule).

\section{Preliminary computations of tensor products}\label{s3}

\subsection{Tensoring band bimodules}\label{s3.1}
The aim of this subsection is to prove the following explicit result.

\begin{proposition}\label{prop21}
For all possible values of parameters, we have
\begin{displaymath}
B(k_1,m_1,\lambda_1)\otimes_A B(k_2,m_2,\lambda_2)\cong
\bigoplus_{s}B(k_1+k_2,s,\lambda_1\lambda_2),
\end{displaymath}
where $s$ runs through the set 
\begin{equation}\label{eq5}
\{|m_1-m_2|+1,|m_1-m_2|+3,\dots,m_1+m_2-1\}.
\end{equation}
\end{proposition}

To prove Proposition~\ref{prop21}, we need some preparation. For any $A$-$A$-bimodule
$X$ and any automorphism $\varphi$ of $A$ we denote by ${}^{\varphi}X$ and
$X^{\varphi}$ the bimodule obtained from $X$ by twisting the left (resp. right)
action of $A$ by $\varphi$. Note that, for indecomposable $X$, both ${}^{\varphi}X$ and
$X^{\varphi}$ are indecomposable.

\begin{lemma}\label{lem22}
For any $A$-$A$-bimodules $X$ and $Y$, we have
$X^{\varphi}\otimes_A Y\cong X\otimes_A {}^{\varphi^{-1}}Y$.
\end{lemma}

\begin{proof}
The correspondence $x\otimes y\mapsto x\otimes y$ from $X^{\varphi}\otimes_A Y$ to 
$X\otimes_A {}^{\varphi^{-1}}Y$ induced a  well-defined map due to the fact that 
$x\varphi(a)\otimes y=x\otimes ay=x\otimes \varphi^{-1}(\varphi(a))y$, 
for all $x\in X$, $y\in Y$ and $a\in A$.
Being well defined, this map is, obviously, an isomorphism of bimodules.
\end{proof}

\begin{lemma}\label{lem23}
For all possible values of parameters, we have
\begin{displaymath}
{}^{\theta}B(k,m,\lambda)\cong B(k-1,m,\lambda)
\quad\text{ and }\quad{}^{\eta_\mu}B(k,m,\lambda)\cong B(k,m,\lambda\mu^{-1}). 
\end{displaymath}
\end{lemma}

\begin{proof}
An indecomposable band $A$-$A$-bimodule $X$ is uniquely defined by its {\em dimension vector}
(as a representation of \eqref{eq1}) together with the trace of the linear operator
\begin{displaymath}
R:=(\cdot\alpha_{j-1}^{-1})(\alpha_n\cdot)(\cdot\alpha_{j-2}^{-1})\dots 
(\alpha_2\cdot)(\cdot\alpha_j^{-1})(\alpha_1\cdot)  
\end{displaymath}
on the vector space $X_{1|j}$, where $j$ denotes the unique element  in $\mathbb{Z}_n$ for 
which we have $\alpha_1\cdot X_{1|j}\neq 0$.

Now, the isomorphism ${}^{\theta}B(k,m,\lambda)\cong B(k-1,m,\lambda)$ follows
by comparing the dimension vectors of $B(k,m,\lambda)$ and $B(k-1,m,\lambda)$
using the definition of $\theta$ and also noting that the twist by $\theta$
does not affect the trace of $R$.

Similarly, the isomorphism $\quad{}^{\eta_\mu}B(k,m,\lambda)\cong B(k,m,\lambda\mu^{-1})$
follows by noting that the twist by $\eta_\mu$ does not affect the dimension vector
but it does multiply the eigenvalue of $R$ by $\mu^{-1}$. The claim follows.
\end{proof}

\begin{proof}[Proof of Proposition~\ref{prop21}.]

Using Lemmata~\ref{lem22} and \ref{lem23},  in the product 
\begin{displaymath}
B(k_1,m_1,\lambda_1)\otimes_A B(k_2,m_2,\lambda_2) 
\end{displaymath}
we can move all involved twists by automorphism
to the right. Hence the claim of Proposition~\ref{prop21}
reduces to the case
\begin{equation}\label{eq6}
B(1,m_1,1)\otimes_A B(1,m_2,1)\cong\bigoplus_{s}B(1,s,1),
\end{equation}
where $s$ runs through \eqref{eq5}. The latter formula should be compared
to the classical formula for the tensor product of Jordan cells,
see \cite[Theorem~2]{DH},
\begin{displaymath}
J_{m_1}(1)\otimes_{\Bbbk} J_{m_2}(1)\cong \bigoplus_{s}J_{s}(1)
\end{displaymath}
where $s$ runs through \eqref{eq5}; and also to the formula for tensoring
simple finite dimensional $\mathfrak{sl}_2$-modules, see e.g. \cite[Theorem~1.39]{Ma}.

This comparison implies that, by induction, \eqref{eq6} reduces to the case
\begin{displaymath}
B(1,2,1)\otimes_A B(1,m,1)\cong 
\begin{cases}
B(1,2,1),& m=1;\\
B(1,m-1,1)\oplus B(1,m+1,1), & m>1;
\end{cases}
\end{displaymath}
and the symmetric case $B(1,m,1)\otimes_A B(1,2,1)$. The latter is similar to the former and
left to the reader. The case $m=1$ is obvious as $B(1,1,1)$ is the regular $A$-$A$-bimodule.
We now consider the case $m>1$. We also assume $n>1$, the case $n=1$
is similar but (as we have already seen several time) requires a change of notation
as, for example, the dimension vector of the regular bimodule does not fit into the
$n>1$ pattern.

Let $a_{i|i}^{(1)}$, $a_{i|i}^{(2)}$, $a_{i+1|i}^{(1)}$, $a_{i+1|i}^{(2)}$,
where $i\in\mathbb{Z}_n$, be the elements of the standard basis of $B(1,2,1)$ 
at the vertices $i|i$ and $i+1|i$, respectively (the value of the bimodule $B(1,2,1)$
at all other verities is zero). Similarly, let 
$b_{i|i}^{(s)}$ and $b_{i+1|i}^{(s)}$, for $s=1,2,\dots,m$,
be elements of the standard basis of $B(1,m,1)$. Then a basis of $B(1,2,1)\otimes_A B(1,m,1)$,
which we will call {\em standard}, is given by all elements of the form 
\begin{equation}\label{eq7}
a_{i|i}^{(1)}\otimes b_{i|i}^{(s)},\quad
a_{i|i}^{(2)}\otimes b_{i|i}^{(s)},\quad
a_{i+1|i}^{(1)}\otimes b_{i|i}^{(s)},\quad
a_{i+1|i}^{(2)}\otimes b_{i|i}^{(s)}.
\end{equation}

That the left action of every $\alpha_i$ in this basis is given by the identity operator
follows from the corresponding property for $B(1,2,1)$. A similar, but slightly more involved
observation is that the right action of every $\alpha_i$, where $i\neq 1$, is also given by 
the identity operator. This follows from the corresponding property for $B(1,m,1)$ and also
the fact that the left such $\alpha_i$ on any side of both $B(1,2,1)$ and $B(1,m,1)$
is given by the identity operator.

It remains to compute the right action of $\alpha_1$. It is given by $J_2(1)$ on $B(1,2,1)$
and by $J_m(1)$ on $B(1,m,1)$. To compute the right action on an element of the form
$a_{2|2}^{(t)}\otimes b_{2|2}^{(s)}$, we first apply the right action on $B(1,m,1)$
which acts on the $b_{2|2}^{(s)}$-component via $J_m(1)$. The outcome is a linear combination
of elements of the form $a_{2|2}^{(t)}\otimes b_{2|1}^{(s)}$ which now have to be
rewritten in the basis \eqref{eq7}. For this we write each $b_{2|1}^{(s)}$ as
$\alpha_1\cdot b_{1|1}^{(s)}$ and move $\alpha_1$ through the tensor product sign.
The element $\alpha_1$ acts on $a_{2|2}^{(t)}$ from the right via $J_2(1)$.
Altogether, we get exactly the Kronecker product of Jordan cells $J_2(1)\otimes_{\Bbbk}J_m(1)$
written in the basis \eqref{eq7}. Using the classical decomposition of the
latter, see \cite[Theorem~2]{DH}, the necessary result follows.
\end{proof}

\subsection{Tensoring string bimodules with band bimodules}\label{s3.2}
The aim of this subsection is to prove the following result.

\begin{proposition}\label{prop26}
Let $X$ be a string $A$-$A$-bimodule and $Y$ a band $A$-$A$-bimodule.

\begin{enumerate}[$($a$)$]
\item\label{prop26.1}
The bimodule $X\otimes_A Y$ decomposes as a direct sum of string $A$-$A$-bimodules,
moreover, both the type, height and width of any summand coincides with that of $X$.
\item\label{prop26.2}
The bimodule $Y\otimes_A X$ decomposes as a direct sum of string $A$-$A$-bimodules,
moreover, both the type, height and width of any summand coincides with that of $X$.
\end{enumerate} 
\end{proposition}

\begin{proof}
We prove the first claim, the proof of the second claim is similar.
Let us start be noting that both $X^{\theta}$ and $X^{\eta_\lambda}$, for any $\lambda \in\Bbbk\setminus\{0\}$, 
are obviously string bimodules of the same height as $X$. Therefore, from Proposition~\ref{prop21}
it follows that it is enough to consider the case $Y=B(1,2,1)$. We use the same notation for
the standard basis in $B(1,2,1)$ as in the proof of Proposition~\ref{prop21}.

Let $\{v_{i|j}\}$ be the standard basis of $X$, where $v_{i|j}$ is a basis of $\Bbbk_{i|j}$
whenever the space in the vertex $i|j$ of \eqref{eq3} is non-zero. Then a basis in 
$X\otimes_A Y$ is given by all elements of the form
\begin{displaymath}
v_{i|j} \otimes a_{j|j}^{(1)}\qquad\text{ and }\qquad
v_{i|j} \otimes a_{j|j}^{(2)}.
\end{displaymath}
Similarly to the proof of Proposition~\ref{prop21}, all non-zero left actions of all $\alpha_i$
in this basis are given by the appropriate identity operators. Similarly for the right actions
of all $\alpha_i$, where $i\neq 1$. The action of $\alpha_1$ is give by a direct sum of 
copies of $J_2(1)$. The picture, written as a representation of \eqref{eq3}, looks as follows:
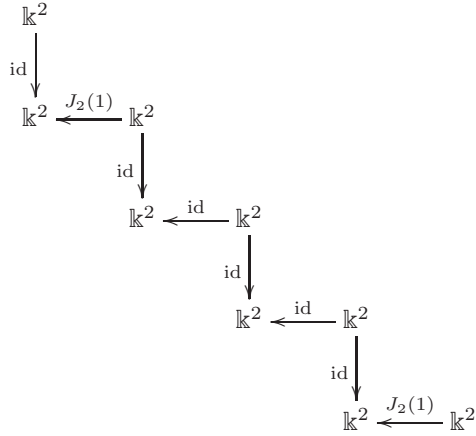
\begin{displaymath}
\xymatrix{
\Bbbk^2\ar[d]_{\mathrm{id}}&&&&&\\
\Bbbk^2&\Bbbk^2\ar[d]_{\mathrm{id}}\ar[l]_{J_2(1)}&&&&\\
&\Bbbk^2&\Bbbk^2\ar[d]_{\mathrm{id}}\ar[l]_{\mathrm{id}}&&&\\
&&\Bbbk^2&\Bbbk^2\ar[d]_{\mathrm{id}}\ar[l]_{\mathrm{id}}&\\
&&&\Bbbk^2&\Bbbk^2\ar[l]_{J_2(1)}\\
}  
\end{displaymath}
It is clear that we can choose new bases in all spaces such that in these new bases all actions
are given by the identity operators. The claim follows.
\end{proof}

\subsection{Tensoring string bimodules}\label{s3.3}
The aim of this subsection is to prove the following result.

\begin{proposition}\label{prop27}
Let $X$ and $Y$ be string $A$-$A$-bimodules. Then the bimodule
$X\otimes_A Y$ decomposes as a direct sum of string and $\Bbbk$-split $A$-$A$-bimodules, 
moreover, both the number of valleys and the width of each string summand does 
not exceed that of $Y$, and both the number of valleys and 
the height of each string summand does not exceed that of $X$.
\end{proposition}

\begin{proof}
We view both $X$ and $Y$ as representations of \eqref{eq3}. Let $x_{i|j}$
be the standard basis of $X$ and $y_{i|j}$ be the standard basis of $Y$.
Then $x_{i|j}\otimes y_{s|t}$ is a basis of $X\otimes_{\Bbbk}Y$.

Consider a directed graph $\Gamma$ whose vertices are all $x_{i|j}\otimes y_{s|t}$ and
arrows are defined as follows:
\begin{itemize} 
\item there is a (vertical) arrow from $x_{i|j}\otimes y_{s|t}$ to $x_{i+1|j}\otimes y_{s|t}$
if $x_{i+1|j}=\alpha_i x_{i|j}$ in $X$.
\item there is a (horizontal) arrow from $x_{i|j}\otimes y_{s|t}$ to $x_{i|j}\otimes y_{s|t-1}$
if $y_{s|t-1}=y_{s|t}\alpha_{t-1}$ in $Y$.
\end{itemize}

The graph $\Gamma$ represents the action of $A$ (on the left via vertical arrows
and on the right via horizontal arrows) on the $\Bbbk$-split 
$A$-$A$-bimodule $X\otimes_{\Bbbk}Y$. Therefore each connected component of 
$\Gamma$ looks as follows:
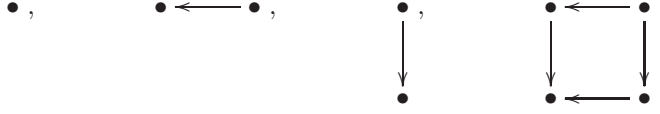
\begin{displaymath}
\xymatrix{\bullet},\qquad\qquad
\xymatrix{\bullet&\bullet\ar[l]},\qquad\qquad
\xymatrix{\bullet\ar[d]\\\bullet},\qquad\qquad
\xymatrix{\bullet\ar[d]&\bullet\ar[d]\ar[l]\\\bullet&\bullet\ar[l]}.
\end{displaymath}
We would like to understand what happens with the graph $\Gamma$ under the projection
$X\otimes_{\Bbbk}Y\tto X\otimes_{A}Y$, that is under factoring out the relations
$xa\otimes y=x\otimes ay$, for $x\in X$, $y\in Y$ and $a\in A$.

First we note that, if $j\not\equiv s\text{ (mod $n$)}$, then $x_{i|j}\otimes y_{s|t}=0$
in $X\otimes_{A}Y$. As $j$ and $s$ are the same for all vertices in a connected
component of $\Gamma$, we can just throw away all connected components for which 
$j\not\equiv s\text{ (mod $n$)}$.

Next consider the valleys in $\Gamma$. These exist only for rectangular connected
components (the last one in the list above). Let $x_{i|j}\otimes y_{s|t}$ be such a valley.
Then one of the following two possible cases occur.

{\bf Case~1.} Neither $x_{i|j}$ nor $y_{s|t}$ is a valley of the graphs of $X$ and $Y$,
respectively. This happens exactly when $X$ is of type $M$ or $S$ and $x_{i|j}$ is the 
``last'', that is south-east vertex in $X$, and $Y$ is of type $M$ or $N$ and $x_{i|j}$ is 
the  ``first'', that is north-west vertex in $Y$. In this case, it is easy to see that 
the whole connected component survives the projection onto $X\otimes_{A}Y$ and gives
there a copy of a ($\Bbbk$-split) projective-injective bimodule.

{\bf Case~2.} At least one of $x_{i|j}$ or $y_{s|t}$ is a valley in the respective
graphs of $X$ and $Y$. Let us consider the case when $x_{i|j}$ is a valley (the other
case is similar). Then $x_{i|j}=x_{i|j+1}\alpha_{j}$. At the same time,
$\alpha_{j}y_{s|t}=0$ as $y_{s|t}$ is in the image of the right action of $A$
(and the left action annihilates the image of the right action on string bimodules).
Therefore such a valley of $\Gamma$ is mapped to zero in $X\otimes_{A}Y$.

A similar argument show that, apart from some vertices of $\Gamma$ being sent to zero
via the projection $X\otimes_{\Bbbk}Y\tto X\otimes_{A}Y$, one can also have
identifications in which the left vertex of one of the components of the form
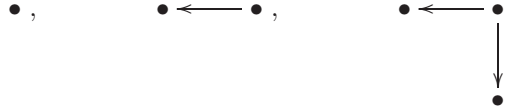
\begin{displaymath}
\xymatrix{\bullet},\qquad\qquad
\xymatrix{\bullet&\bullet\ar[l]},\qquad\qquad
\xymatrix{\bullet&\bullet\ar[d]\ar[l]\\&\bullet}
\end{displaymath}
(here the last component is obtained from a rectangular component of $\Gamma$ the valley
of which gets killed) is identified with the bottom vertex of 
one of the components of the form
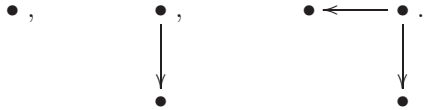
\begin{displaymath}
\xymatrix{\bullet},\qquad\qquad
\xymatrix{\bullet\ar[d]\\\bullet},\qquad\qquad
\xymatrix{\bullet&\bullet\ar[d]\ar[l]\\&\bullet}.
\end{displaymath}
This implies that the action graph of $A$ on any non $\Bbbk$-split 
direct summand of  $X\otimes_{A}Y$ is a tree and hence this
direct summand must be a string bimodule.

It remains to prove the claims about the width and the height. We prove the first
one, the second one is similar. Note that the width of any connected component of 
$\Gamma$ is at most one. To get a summand of greater width, we need the identifications
as described in the previous paragraph to happen between components of the form
\begin{equation}\label{eq8}
\xymatrix{\bullet&\bullet\ar[l]},\qquad\qquad
\xymatrix{\bullet&\bullet\ar[d]\ar[l]\\&\bullet}
\end{equation}
where the first one can appear at most ones, on the far right as it appearance 
stops farther possible identifications on the right. As follows from the above, potential
identifications correspond to valleys of $Y$. Indeed, we need a fragment in
$Y$ of the form  $\xymatrix{\bullet&\ar[l]}$ to have one of the components in
\eqref{eq8} and we need a fragment in $Y$ of the form
\begin{displaymath}
\xymatrix{\ar[d]\\\bullet}
\end{displaymath}
for the same vertex in oder to be able to do identification. Putting all this together,
we see that both the width and the number of valleys cannot increase. The claim follows.
\end{proof}

\subsection{Some explicit computations}\label{s3.4}

\begin{lemma}\label{lem31}
For any $k>0$, the bimodule $M_{1|2}^{(k-1)}$ is a direct 
summand of the bimodule $M_{1|1}^{(k)}\otimes_A  M_{1|1}^{(k)}$.
\end{lemma}

\begin{proof}
Let $\{x_{i|j}\}$ be the standard basis of  $M_{1|1}^{(k)}$ considered as a representation
of \eqref{eq3}. Following the arguments in the proof of Proposition~\ref{prop27},
we see that $M_{1|1}^{(k)}\otimes_A  M_{1|1}^{(k)}$ contains a direct summand which
is the span of the following elements:
\begin{multline*}
x_{1|1}\otimes x_{1|2}=x_{1|2}\otimes x_{2|2},\quad
x_{1|2}\otimes x_{2|3},\quad
x_{2|2}\otimes x_{2|3}=x_{2|3}\otimes x_{3|3},\dots,\\
x_{k+1|k+1}\otimes x_{k+1|k+2}=x_{k+1|k+2}\otimes x_{k+2|k+2}.
\end{multline*}
This direct summand is isomorphic to $M_{1|2}^{(k-1)}$, the claim follows.
\end{proof}

\begin{lemma}\label{lem32}
For any $k>0$, the bimodule $N_{1|1}^{(k)}$ is a direct 
summand of the bimodule $W_{1|1}^{(k)}\otimes_A  M_{1|1}^{(k)}$.
\end{lemma}

\begin{proof}
Let $\{x_{i|j}\}$ be the standard basis of  $W_{1|1}^{(k)}$ 
and $\{y_{i|j}\}$ be the standard basis of  $M_{1|1}^{(k)}$, both 
considered as a representation of \eqref{eq3}. 
Following the arguments in the proof of Proposition~\ref{prop27},
we see that $W_{1|1}^{(k)}\otimes_A  M_{1|1}^{(k)}$ contains a direct summand which
is the span of the following elements:
\begin{multline*}
x_{1|1}\otimes y_{1|1},\quad 
x_{1|1}\otimes y_{1|2},\quad
x_{2|1}\otimes y_{1|2}=x_{2|2}\otimes y_{2|2},\\
x_{2|2}\otimes y_{2|3},\quad\dots,\quad
x_{k+1|k+1}\otimes y_{k+1|k+2}.
\end{multline*}
This direct summand is isomorphic to $N_{1|1}^{(k)}$, the claim follows.
\end{proof}

\begin{lemma}\label{lem33}
For any $k>0$, the bimodule $M_{1|1}^{(k)}$ is a direct 
summand of the bimodule $S_{1|1}^{(k)}\otimes_A  N_{1|1}^{(k)}$.
\end{lemma}

\begin{proof}
Let $\{x_{i|j}\}$ be the standard basis of  $S_{1|1}^{(k)}$ 
and $\{y_{i|j}\}$ be the standard basis of  $N_{1|1}^{(k)}$, both 
considered as a representation of \eqref{eq3}. 
Following the arguments in the proof of Proposition~\ref{prop27},
we see that $S_{1|1}^{(k)}\otimes_A  N_{1|1}^{(k)}$ contains a direct summand which
is the span of the following elements:
\begin{multline*}
x_{1|1}\otimes y_{1|1},\quad 
x_{1|1}\otimes y_{1|2},\quad
x_{2|1}\otimes y_{1|2}=x_{2|2}\otimes y_{2|2},\\
x_{2|2}\otimes y_{2|3},\quad\dots,\quad
x_{k+1|k+1}\otimes y_{k+1|k+2},\quad 
x_{k+2|k+1}\otimes y_{k+1|k+2}.
\end{multline*}
This direct summand is isomorphic to $M_{1|1}^{(k)}$, the claim follows.
\end{proof}

\begin{lemma}\label{lem34}
For any $k>0$, the bimodule $W_{1|1}^{(k)}$ is a direct 
summand of the bimodule $N_{1|1}^{(k)}\otimes_A  S_{2|1}^{(k)}$.
\end{lemma}

\begin{proof}
Let $\{x_{i|j}\}$ be the standard basis of  $N_{1|1}^{(k)}$ 
and $\{y_{i|j}\}$ be the standard basis of  $S_{1|1}^{(k)}$, both 
considered as a representation of \eqref{eq3}. 
Following the arguments in the proof of Proposition~\ref{prop27},
we see that $N_{1|1}^{(k)}\otimes_A  S_{2|1}^{(k)}$ contains a direct summand which
is the span of the following elements:
\begin{displaymath}
x_{1|2}\otimes y_{2|1},\quad 
x_{2|2}\otimes y_{2|1}=x_{2|3}\otimes y_{3|1},\quad 
x_{2|3}\otimes y_{3|2},\quad\dots,
x_{k+1|k+2}\otimes y_{k+2|k+1}.
\end{displaymath}
This direct summand is isomorphic to $W_{1|1}^{(k)}$, the claim follows.
\end{proof}

\begin{lemma}\label{lem35}
For any $k>0$, the bimodule $S_{1|2}^{(k)}$ is a direct 
summand of the bimodule $M_{1|1}^{(k)}\otimes_A  W_{2|2}^{(k)}$.
\end{lemma}

\begin{proof}
Let $\{x_{i|j}\}$ be the standard basis of  $M_{1|1}^{(k)}$ 
and $\{y_{i|j}\}$ be the standard basis of  $W_{2|2}^{(k)}$, both 
considered as a representation of \eqref{eq3}. 
Following the arguments in the proof of Proposition~\ref{prop27},
we see that $M_{1|1}^{(k)}\otimes_A  W_{2|2}^{(k)}$ contains a direct summand which
is the span of the following elements:
\begin{displaymath}
x_{1|2}\otimes y_{2|2},\quad 
x_{2|2}\otimes y_{2|2}=x_{2|3}\otimes y_{3|2},\quad 
x_{2|3}\otimes y_{3|3},\quad\dots,
x_{k+2|k+2}\otimes y_{k+2|k+2}.
\end{displaymath}
This direct summand is isomorphic to $S_{1|2}^{(k)}$, the claim follows.
\end{proof}

\section{Proof of Theorem~\ref{thm1}}\label{s4}

\subsection{Proof of Theorem~\ref{thm1}\eqref{thm1.1}}\label{s4.1}
This follows immediately from Propositions~\ref{prop21}, \ref{prop26} and \ref{prop27}.

\subsection{Proof of Theorem~\ref{thm1}\eqref{thm1.2}}\label{s4.2}
After Propositions~\ref{prop27}, we only need to argue that all 
string bimodules with the same number of valleys belong to 
the same two-sided cell. If we fix a type, then, looking at
the action graph of the bimodule, it is clear that all bimodules
with fixed number of valleys of this particular type can
be transformed into each other by twisting by some powers of
$\theta$ on both sides. Twisting by $\theta$ is the same as 
tensoring with $A^{\theta}$ or ${}^{\theta}\hspace{-1mm}A$.
Therefore the fact that all bimodules
of the same type having the same number of valleys belong to the same
two-sided cell follows from Propositions~\ref{prop27}.

Lemmata~\ref{lem32} and \ref{lem33} imply that the bimodules of type $M$ and $N$
with the same number of valleys belong to the same two-sided cell.
Lemmata~\ref{lem32} and \ref{lem34} imply that the bimodules of type $W$ and $N$
with the same number of valleys belong to the same two-sided cell.
Finally, Lemmata~\ref{lem33} and \ref{lem35} imply that the bimodules of type $M$ and $S$
with the same number of valleys belong to the same two-sided cell.

\subsection{Proof of Theorem~\ref{thm1}\eqref{thm1.3}}\label{s4.3}
This is proved in Subsection~\ref{s2.2}.

\subsection{Proof of Theorem~\ref{thm1}\eqref{thm1.4}}\label{s4.4}
All non-$\Bbbk$-split string bimodules with zero valleys are of type $M$.
Therefore the fact that they all form a two-sided cell follows from 
Propositions~\ref{prop21}, \ref{prop26} and \ref{prop27} and the
observation in the first paragraph of Subsection~\ref{s4.2}.

\subsection{Proof of Theorem~\ref{thm1}\eqref{thm1.5}}\label{s4.5}
It is clear that $\mathcal{J}_{\mathrm{band}}$ is the minimum and
$\mathcal{J}_{\mathrm{split}}$ is the maximum two-sided cells. Hence,
we just need to prove that $\mathcal{J}_{k-1}\geq_J\mathcal{J}_k$.
However, this follows from Lemma~\ref{lem31}.

\section{Proof of Theorem~\ref{thm2}}\label{s5}

\subsection{Proof of Theorem~\ref{thm2}\eqref{thm2.1}}\label{s5.1}
This follows from Proposition~\ref{prop21}.

\subsection{Proof of Theorem~\ref{thm2}\eqref{thm2.2}}\label{s5.2}
This is proved in Subsection~\ref{s2.2}.

\subsection{Proof of Theorem~\ref{thm2}\eqref{thm2.3}}\label{s5.3}
This is proved in Subsection~\ref{s2.2}.

\subsection{Proof of Theorem~\ref{thm2}\eqref{thm2.4}}\label{s5.4}

Let $X$ be a string representation of \eqref{eq3}. The {\em right support}
of $X$ is the set of all $j\in\mathbb{Z}$ for which there is $i\in\mathbb{Z}$
such that the value of $X$ at the vertex $i|j$ is non-zero. The {\em left support}
of $X$ is the set of all $i\in\mathbb{Z}$ for which there is $j\in\mathbb{Z}$
such that the value of $X$ at the vertex $i|j$ is non-zero. Both the left and 
the right support of a string module is an integer interval. The width of a 
module is simply the cardinality of the right support of this module plus one. 
The height of a  module is simply the cardinality of the left support of this 
module plus one. Therefore the initial vertex and the width uniquely determine
the right support. The initial vertex and the height uniquely determine
the left support.

If $Y$ is a string $A$-$A$-bimodule and $X$ is any $A$-$A$-bimodule, then, by
Propositions~\ref{prop26} and \ref{prop27}, every direct summand of $X\otimes_A Y$
is a string $A$-$A$-bimodule. Moreover, directly from the definitions it follows
that the right support of each such direct summand is a subset of the right support
of $Y$. Therefore, to be in the same left cell, two string $A$-$A$-bimodules must
have the same left support. By Theorem~\ref{thm1}\eqref{thm1.4}, the two bimodules
should also have the same number of valleys. This implies that they either
have the same type of  of type $M$ or $N$ alternatively of type $W$ or $S$.

If two string $A$-$A$-bimodules $X$ and $Y$ are of the same type and have 
the same right support, then $X\cong{}^{\theta^k}Y$, for some power of $k$
and hence they belong to the same right cell by the same argument as
in the first paragraph of Subsection~\ref{s4.2}. From 
Lemmata~\ref{lem32} and \ref{lem33} it thus follows that two bimodules 
of type $M$ and $N$ and with the same right support belong to the same left cell.
From Lemmata~\ref{lem34} and \ref{lem35} it follows that two bimodules 
of type $W$ and $S$ and with the same right support belong to the same left cell.

\subsection{Proof of Theorem~\ref{thm2}\eqref{thm2.5}}\label{s5.5}

Mutatis mutandis the proof of Theorem~\ref{thm2}\eqref{thm2.4}.

\subsection{Proof of Theorem~\ref{thm2}\eqref{thm2.6}}\label{s5.6}

Strong regularity of a two-sided cell $\mathcal{J}$ in the sense of 
\cite[Subsection~4.8]{MM1} means that the intersection of any left and
any right cell inside $\mathcal{J}$ consists of one element. 
For $\mathcal{J}_{\mathrm{split}}$, the claim follows directly from
the description of left and right cells in Theorem~\ref{thm2}\eqref{thm2.2} and
Theorem~\ref{thm2}\eqref{thm2.3}. For $\mathcal{J}_{k}$,
the claim follows from the description of left and 
right cells in Theorem~\ref{thm2}\eqref{thm2.4} and Theorem~\ref{thm2}\eqref{thm2.5}.

\vspace{5mm}

\noindent

Department of Mathematics, Uppsala University, Box. 480,

SE-75106, Uppsala, SWEDEN, email: {\tt helena.jonsson\symbol{64}math.uu.se}

\end{document}